\documentclass[11pt,reqno]{amsproc}

\numberwithin{equation}{section}
\usepackage{todonotes}

\usepackage{verbatim}
\usepackage{times}
\usepackage{amsmath,amsfonts,amstext,amssymb,amsbsy,amsopn,amsthm,eucal}
\usepackage{txfonts}
\usepackage{dsfont}
\usepackage{graphicx}   
\usepackage{hyperref}
\usepackage{accents}
\usepackage{enumerate}
\usepackage{xcolor}

\newcommand{\RR}{\mathds{R}}

\newcommand{\Rm}{{\rm Rm}}
\newcommand{\Ric}{{\rm Ric}}

\newcommand{\Vol}{{\rm Vol}}
\newcommand{\diam}{{\rm diam}}

\newcommand{\rv}{{\rm v}}

\newcommand{\cB}{\mathcal{B}}

\newcommand{\cF}{\mathcal{F}}

\newcommand{\cM}{\mathcal{M}}
\newcommand{\cN}{\mathcal{N}}

\newcommand{\cV}{\mathcal{V}}
\newcommand{\cVM}{\mathcal{VM}}

\newcommand{\cVR}{\mathcal{VR}}

\newtheorem{theorem}{Theorem}[section]

\newtheorem{proposition}[theorem]{Proposition}
\newtheorem{lemma}[theorem]{Lemma}
\newtheorem{corollary}[theorem]{Corollary}
\theoremstyle{definition}
\newtheorem{definition}[theorem]{Definition}
\theoremstyle{remark}
\newtheorem{remark}{Remark}[section]
\theoremstyle{remark}

\theoremstyle{remark}

\theoremstyle{remark}\newtheorem{conjecture}{Conjecture}[section]
\theoremstyle{remark}\newtheorem{question}[conjecture]{Question}

\begin{document}

\title{Finite Diffeomorphism Theorem for manifolds with lower Ricci curvature and bounded energy}

\author{Wenshuai Jiang }
\address{School of Mathematical Sciences\\
Zhejiang University,  Hangzhou 310058, China}
\email{wsjiang@zju.edu.cn}
\thanks{WJ was partially supported by National Key Research and Development Program of China (No. 2022YFA1005501), NSFC (No. 12125105 and No. 12071425) and the Fundamental Research Funds
for the Central Universities 2022FZZX01-01.}
\author{Guofang Wei}
\address{Department of Mathematics\\
	University of California, Santa Barbara\\
	Santa Barbara, CA 93106}
\email{wei@math.ucsb.edu}
\thanks{ GW was partially supported by NSF DMS grant 2104704.}

\dedicatory{Dedicated to Jeff Cheeger's 80th Birthday}
\begin{abstract}
In this paper we prove that the space $\cM(n,\rv,D,\Lambda):=\{(M^n,g) \text{ closed }:  ~~\Ric\ge -(n-1),~\Vol(M)\ge \rv>0, \diam(M)\le D \text{ and }   \int_{M}|\Rm|^{n/2}\le \Lambda\}$ has at most $C(n,\rv,D,\Lambda)$ many diffeomorphism types. This removes the upper Ricci curvature bound of Anderson-Cheeger's finite diffeomorphism theorem in \cite{AnCh}. Furthermore, if $M$ is K\"ahler surface, the Riemann curvature $L^2$ bound could be replaced by the scalar curvature $L^2$ bound. 
\end{abstract}

\maketitle


\section{Introduction}

In studying the relation between geometry and topology, one of the important themes is sphere theorems and finiteness theorems. In \cite{Cheeger_finiteness} Cheeger obtained his famous finiteness theorem that Riemannian manifolds $M^n$ satisfying 
\[
|K_M | \le K,~\Vol(M)\ge \rv>0, \ \diam(M)\le D 
\]
have finitely many diffeomorphism types. Here $K_M$ is the sectional curvature of $M$. The key is showing a positive uniform lower bound on the injectivity radius for the class. Later \cite{Grove-Petersen} Grove-Petersen proved finitely many homotopy types for above class without curvature upper bound by controlling the contractibility radius. The corresponding result replacing sectional curvature lower with Ricci lower bound is only true when $n=3$, \cite{Zhu93}. When $n \ge 4$ counterexamples are given in \cite{Perelman97}.
With additional integral curvature assumption, 
Anderson-Cheeger \cite{AnCh} proved the following class of closed manifolds $M^n$
\begin{align}
|\Ric| \le n-1,~\Vol(M)\ge \rv>0, \diam(M)\le D \text{ and }   \int_{M}|\Rm|^{n/2}\le \Lambda  \label{Ricci-bothsides}
\end{align}
have finitely many  diffeomorphism types. This finiteness result is new in the sense that its Gromov-Hausdorff limit may not be a topological manifold. Hence it can not be proven by Ricci flow smoothing, instead it is done by bubble decomposition. 

Note also that the class of manifolds satisfying $|\Ric| \le n-1,~\Vol(M)\ge \rv>0, \diam(M)\le D$ automatically has a
$L^2$ integral curvature bound \cite{Jiang_Naber}. Hence in dimension 4, no extra integral curvature is needed \cite{CheegerNaber_codimension4}. 

Anderson-Cheeger's finiteness result is extended recently in \cite{Qian23}, where the two sided Ricci curvature bounds is replaced by an $L^p$ bound on the Ricci curvature with $p > n/2$, but with all scale non-collapsing condition. In this paper we extend Anderson-Cheeger's finiteness result by removing the Ricci curvature upper bound. 

 Let us denote $$\cM(n,\rv,D,\Lambda):=\{(M^n,g) \text{ closed }:  ~~\Ric\ge -(n-1),~\Vol(M)\ge \rv>0, \diam(M)\le D \text{ and }   \int_{M}|\Rm|^{n/2}\le \Lambda\}.$$  We have
\begin{theorem}\label{t:maintheorem}
For any given $n,\rv,D,\Lambda >0$, the space $\cM(n,\rv,D,\Lambda)$ has at most $C(n,\rv,D,\Lambda)$ many diffeomorphism types. 
\end{theorem}
Under the assumption that the $L^{n/2}$ curvature tensor is small depending $n,v, D$ instead of boundedness in above the finitely many homeomorphism types was obtained in \cite{Jin18}. Our result greatly improves the result in \cite{Jin18}, also generalizes the finiteness result in \cite{Qian23}, see Remark~\ref{r:int}. 

As we mentioned earlier  without the integral bound on the curvature tensor,  this is true \cite{Zhu93}  when $n=3$,  but it is false when $n \ge 4$ \cite{Perelman97}. 

Note that the integral curvature $\int_{M}|\Rm|^{n/2}$ is scale invariant, and the power $n/2$ is the critical case. With $L^p$ and $p>n/2$, one can obtains the finiteness  for the class by Ricci flow smoothing, see e.g. \cite{DPW00}. For $p < n/2$ the result is not true anymore. 

\begin{remark}  \label{r:int}
With the extension of the Cheeger-Colding theory for manifolds with Ricci curvature lower bound to Ricc lower bound in $L^p$ sense for $p > n/2$ \cite{PW01, TZ}, Theorem~\ref{t:maintheorem} is true if we replace the Ricci lower bound by,  with $p>n/2$, the normalized $L^p$ smallness of Ricci curvature below $-(n-1)$, or by boundedness of $L^p$ of negative part of Ricc curvature and replace global volume lower bound by all scale volume noncollapsing, hence extending the finiteness result in \cite{Qian23}. 
\end{remark}

In the four dimensional K\"ahler case, we can replace $L^2$ bound on the curvature tensor by $L^2$ bound on the scalar curvature. 

\begin{theorem}\label{t:Kahler_diffeomorphism}
Let $(M^4,g)$ be compact K\"ahler manifolds with $\Ric\ge -3$, $\Vol(M)\ge \rv>0$, $\diam(M)\le D$ and $\int_{M}|R|^2\le \Lambda$ (Here $R$ is the scalar curvature). Then $(M^4,g)$ have at most $C(\rv,D,\Lambda)$ many diffeomorphism types.
\end{theorem}

Here with the pointwise Ricci curvature lower bound, an $L^2$ bound on the scalar curvature is equivalent to an $L^2$ bound on the Ricci curvature. In Theorem \ref{t:Kahler_diffeomorphism}, without the K\"ahler condition, if we assume $L^p$-Ricci with $p>2$, then based on Cheeger-Naber's argument in \cite{CheegerNaber_codimension4} one can prove the same finitely many diffeomorphism types theorem \cite{Jiang}. 

In proving the finiteness in \cite{AnCh}, one of the key tools is the $\epsilon$-regularity result that for manifolds in \eqref{Ricci-bothsides}, when the $L^{n/2}$-curvature tensor of a ball is sufficiently small, then the harmonic radius of the ball is uniformly positive.  Similarly this is done in \cite{Qian23}. This uses two sided Ricci curvature bound crucially and can not be true with only Ricci curvature lower bound. Instead we use Reifenberg radius (see Definition~\ref{d:DefReifenbergradius}) to control the topology. Another important result in \cite{AnCh} is their neck theorem which provides control over the geometry and
topology of the transition regions.

To prove Theorem~\ref{t:maintheorem}, our start point is
 \cite[Theorem 10.2]{Cheeger}, which gives that the tangent cone of any limit of $M_i^n \in \mathcal{M}(n,\rv,D,\Lambda)$ is unique and is the cone over some space form $\mathbb{S}^{n-1}/\Gamma$ with $|\Gamma|\le C(n,\rv)$. This provide the Reifenberg radius lower bound, see Lemma~\ref{l:cross_section}. 

Under the condition of Theorem \ref{t:Kahler_diffeomorphism}, we do not know whether every tangent cone of the limit space of such sequence is unique and flat. Instead, we prove that every tangent cone is a cone with  smooth cross section and all cross sections have the same topology with uniform lower bound for the Reifenberg radius (See Lemma \ref{l:cross_section_Kahler}). To see this, we introduce a new monotone quantity, \begin{align}
      \label{e:VR montone}
\cVR_r(x):=\cV_r(x)+\int_{B_r(x)}|R|^{2}, ~~x\in M, r>0
\end{align}
where $\cV_r(x)$ is the volume ratio \eqref{e:volume ratio}.
With additional deformation result we are able to glue local homeomorphisms to get a global homeomorphism on the neck region. To control the cross sections we derive an $\epsilon$-regularity for 4-dim K\"ahler manifolds when $L^2$ Ricci curvature is small (see Proposition~\ref{p:eps_Kahler}).  By Cheeger-Colding, we know that without any  curvature integral condition the cross section may not be smooth. Recently, it was proved in \cite{BPS} that the cross section of every tangent cone is homeomorphic to $\mathbb{S}^{3}/\Gamma$ without the curvature integral condition, which solves a conjecture of Colding-Naber.  

For both Theorems~\ref{t:maintheorem} \ref{t:Kahler_diffeomorphism} we construct a decomposition as in \cite{CheegerNaber_codimension4}. For Theorems~\ref{t:maintheorem} we use the monotonic quantity volume ratio \eqref{e:volume ratio}, while for Theorems~\ref{t:Kahler_diffeomorphism} we use the modified monotonic quantity \eqref{e:VR montone}. 

From Theorems~\ref{t:Kahler_diffeomorphism}, it is natural to  ask if it is true without K\"ahler condition. Namely

\begin{question}
For given $D,\Lambda, V$, does there exist $C(D,\Lambda,V)>0$ such that the space $\mathcal{M}(D,\Lambda,V):=\{(M^4,g) \text{ closed}: \text{$\diam(M)\le D, \Vol(M)\ge V, \Ric\ge -3,$ and $\int_{M}|\Ric|^2\le \Lambda$ } \}$  have at most $C(D,\Lambda,V)$ many diffeomorphism type? 
\end{question}
From our proof, one only needs to prove that every tangent cone of the limit space of any sequence of $\mathcal{M}(D,\Lambda,V)$ is a cone over smooth cross section with uniform lower bound for the Reifenberg radius. Once this is true, by Chern-Gauss-Bonnet theorem, the manifold has bounded $L^2$-curvature. 

Acknowledgement: The authors would like to thank Jeff Cheeger for his interest, encouragement and helpful comments. 



\section{Preliminary}
In this section, we will recall some important results from Cheeger-Colding theory \cite{ChC1}.

\subsection{Almost metric cone and almost splitting}
The following two results of Cheeger-Colding play crucial roles in the study of manifolds with  Ricci curvature lower bound. We will use this several times in our proof.

For noncollapsed manifold, almost metric cone structure is very important, which makes a great difference between collapsed and noncollapsed manifolds. Let us start from the following Cheeger-Colding's almost volume cone implies almost metric cone theorem.
\begin{theorem} \cite{ChC1}, see also \cite[Theorem 9.45, Remark 9.66]{Cheeger-note} 
\label{t:Almostmetriccone}
For any $n,\epsilon>0$ there exists $\delta=\delta(n,\epsilon)>0$ such that the following holds: Let $(M^n,g,p)$ be a complete manifold with $\Ric\ge -\delta(n-1)$. If the volume ratio $\cV_r(x):=- \ln \frac{\Vol(B_r(x))}{\Vol_{-\delta}(B_r)}$ satisfies 
\begin{align}
|\cV_2(p)-\cV_1(p)|\le \delta,
\end{align}
then 
\begin{align}
d_{GH}(B_{1}(p),B_{1}(x_c))\le \epsilon,
\end{align}
where $B_{1}(x_c)$ is the ball of a metric cone $(C(X),x_c)$ with cone vertex $x_c$, and $\Vol_{-\delta}(B_r)$ is the $r$-ball volume in the model space $\mathbb M^n_{-\delta}$ with constant sectional curvature $-\delta$.  
\end{theorem}

Let us remark that  there is no volume lower bound assumption for Theorem \ref{t:Almostmetriccone}(see \cite{ChTi05} for a nice application without volume assumption). If we further assume the volume lower bound, by Theorem \ref{t:Almostmetriccone} we have the following
\begin{lemma}\label{l:cone_structure}
Let $(M^n,g,p)$ be complete manifold satisfying $\Ric\ge -(n-1)$ and $\Vol(B_1(p))\ge \rv>0$. Then for any $\delta>0, $ there exists $0<\kappa=\kappa(n,\rv,\delta)\le 1$ such that for any $r\le 1$ and $x\in B_1(p)$, there exists $\kappa r\le r_x\le r$ such that 
\begin{align}\label{e:finitepinchvolumeratio}
|\cV_{r_x/100}(x)-\cV_{r_x}(x)|\le \delta.
\end{align}
and 
\begin{align}
d_{GH}(B_{r_x}(x),B_{r_x}(x_c))\le  \delta r_x,~~~\text{ for a metric cone $(C(X),x_c))$.}
\end{align}
\end{lemma}
Here and in the rest of the paper we denote \begin{align}
    \cV_r(x):=- \ln \frac{\Vol(B_r(x))}{\Vol_{-1}(B_r)}, \label{e:volume ratio}
\end{align}
the volume ratio comparing to the model space $\mathbb H^n$. Clearly $\cV_0(x) =0$ and it is monotonically increasing in $r$. 
\begin{proof}
By Theorem \ref{t:Almostmetriccone}  it suffices to prove  (\ref{e:finitepinchvolumeratio}). 
For any $0<r<1$ and $x\in B_1(p)$, let $r_i=r100^{-i}$. Denote $\cF_i:=|\cV_{r_i}(x)-\cV_{r_{i+1}}(x)|$. Note that $|\cV_1(x)|\le C(n,\rv)$.
  If for all $0\le i\le L:=\left \lfloor {C(n,\rv)/\delta} \right \rfloor +1$, we have $\cF_i>\delta$, then 
\begin{align}
C(n,\rv) <  \delta L\le \sum_{i=0}^{L-1}\cF_i=\sum_{i=0}^{L-1}|\cV_{r_i}(x)-\cV_{r_{i+1}}(x)|=|\cV_r(x)-\cV_{r_L}(x)|\le |\cV_r(x)| \le |\cV_1(x)|\le C(n,\rv).
\end{align}
This is a contradiction. Therefore, the lemma follows with $k = 100^{-L+1}$. 
\end{proof}

 The second result is Cheeger-Colding's almost splitting theorem, which is an quantitative version of Cheeger-Gromoll's splitting theorem. 
\begin{theorem}[Almost splitting, \cite{ChC1}]\label{t:Almostsplitting}
For any $n,\epsilon>0$ there exists $\delta=\delta(n,\epsilon)$ such that the following holds: Let $(M^n,g,p)$ be a complete manifold with $\Ric\ge -\delta(n-1)$. If $\gamma: [-\delta^{-1},\delta^{-1}]\to M$ is a minimizing geodesic with $\gamma(0)=p$, then 
\begin{align}
d_{GH}(B_1(p), B_1(0,x))\le \epsilon,
\end{align}
where $\mathbb{R}\times X$ is product space with $(0,x)\in \mathbb{R}\times X$. 
\end{theorem}

Roughly, by Theorem \ref{t:Almostsplitting} we see that if a ball $B_2(p)$ is close to a metric cone, then for any $q\in B_2(p)\setminus \bar B_1(p)$, definite size ball with center $q$ would be close to a splitting space. Under volume noncollpased condition, by combining Theorem \ref{t:Almostmetriccone}, one can prove more. {See more discussions and applications in \cite{CJN,CheegerNaber_Ricci,CheegerNaber_codimension4,Jiang_Naber}

\subsection{Reifenberg Radius and finiteness theorem}

Let us define the $\epsilon$-Reifenberg radius for a metric space $(X,d)$.
\begin{definition}\label{d:DefReifenbergradius}
    Let $(X,d)$ be a metric space. 
    For $\epsilon>0$, integer $n>0$ and $x\in X$, define the Reifenberg radius at $x$ by
\begin{align}
r_{Rei,\epsilon,n}(x):=\max\{r\ge 0:~~d_{GH}(B_s(y),B_s(0^n))\le \epsilon s,~~\text{ for all $0\le s\le r$ and $y\in B_r(x)$}\}
\end{align}
and define the Reifenberg radius of $X$ by
 \begin{align}
 r_{Rei,\epsilon,n}(X):=\inf_{x\in X}r_{Rei,\epsilon,n}(x).
 \end{align}
 \end{definition}

 \begin{remark}
     If $X$ is a Riemannian manifold with dimension $n$, we will omit the index $n$ in $r_{Rei, \epsilon, n}$ and write it as $r_{Rei, \epsilon}$.
 \end{remark}
 
Let us recall some results proved by Cheeger-Colding \cite{ChC1} under lower Reifenberg radius bound. 
\begin{lemma}\label{l:finite_homeo}[Theorem A.1.4 in \cite{ChC1}]
Given any $r_0>0$, let $\mathcal{M}(n, \epsilon, r_0)$ be the isometry classes of compact metric spaces $(X,d)$ with $\epsilon$-Reifenberg radius $r_{Rei,\epsilon,n}(X)>r_0\,\diam(X,d)>0$. There exists $\epsilon(n)>0$ such that if the $\epsilon < \epsilon(n)$, then $\mathcal{M}(n, \epsilon, r_0)$ has at most $C(n,r_0)$ many homeomorphism types. Moreover, if $X$ is Riemannian, then it has at most $C(n,r_0)$ many diffeomorphism types. 
\end{lemma}

\begin{proposition}\label{p:homeo_XY}[Theorem A.1.3 in \cite{ChC1} ]
Let $(X,d_1)$ and $(Y,d_2)$ be two compact metric spaces. 
Assume the Reifenberg radius $\min\{ r_{Rei,\epsilon,n}(X),r_{Rei,\epsilon,n}(Y)\}\ge r_0>0$. There exists a constant $\epsilon(n)$ such that if $\epsilon\le \epsilon(n)$ and 
\begin{align}
d_{GH}(X,Y)\le \epsilon,
\end{align}
then $X$ is homeomorphic to $Y$. Moreover, if $X,Y$ are Riemannian, then the homeomorphism is a diffeomorphism. 
\end{proposition}
\begin{remark}
$\epsilon$ in Proposition \ref{p:homeo_XY} is independent of $r_0$. Actually, by scaling we can always assume 
$r_0
\ge 1$. 
\end{remark}
For manifolds with Ricci curvature lower bound, let us recall the following Reifenberg radius lower bound estimate from \cite{Colding} \cite{ChC1}. 
\begin{lemma}[\cite{Colding} \cite{ChC1} ]\label{l:Reifenberg_Colding}
Let $(M^n,g,p)$ satisfy $\Ric\ge -(n-1)$. For any $\epsilon>0$, if $\delta\le \delta(n,\epsilon)$ and $d_{GH}(B_1(p),B_1(0^n))\le \delta$, then the Reifenberg radius $r_{Rei,\epsilon,n}(x)\ge 1/4$ for any $x\in B_{1/2}(p)$.
\end{lemma}

\subsection{Diffeomorphism of Body Region and Finite topology}

\begin{lemma}\label{l:finite_topology_bodyregion}
	Let $M_1^n,M_2^n$ be two complete manifolds and $U_i\subset M_i$ are open subsets with Reifenberg radius $r_{Rei,\epsilon}(x)\ge r>0$ for all $x\in U_i$, $i=1,2$. There exists $\epsilon(n)>0$ such that if 
	\begin{align}
	d_{GH}(B_r(U_1),B_r(U_2))\le r\epsilon,
	\end{align}
then there exists a diffeomorphism $F: U_1'\to U_2'$ with $B_{r/2}(U_i)\subset U_i'\subset B_r(U_i)$. 
	\end{lemma}
	\begin{proof}
	The proof of Theorem A.1.2 in \cite{ChC1} can be used here. We will just sketch the proof. For more detail, see \cite{ChC1}. 
 
By rescaling, without loss of generality, we can assume $r=30$. For each $i\ge 0$, we can choose finite subsets $X_i^\alpha\subset B_r(U_\alpha)$ such that $X_i^\alpha$ is a minimal $2^{-i}$-dense subset of $B_r(U_\alpha)$ and $X_0^\alpha\subset X_1^\alpha\subset \cdots $. We can rewrite $X_i^\alpha=\bigcup_{j=1}^{N_i^\alpha}Q_{i,j}^\alpha$ as \cite{ChC1} with $N_i^\alpha\le N(n)$.	 Following the same argument as \cite{ChC1} by constructing and suitably modifying the transition maps, if $\epsilon\le \epsilon(n)$ we can construct a sequence of manifolds $W_i^\alpha$ and diffeomorphisms $f_i^\alpha: W_i^\alpha\to W_{i+1}^\alpha$, and for $i_0$ sufficiently large,  diffeomorphism $f^\alpha: W_{i_0}^\alpha \to U_\alpha'$ with $B_{10}(U_\alpha)\subset U_\alpha'\subset B_{20}(U_\alpha)$. Furthermore, as in \cite{ChC1} if $\epsilon\le \epsilon(n)$, there exists a diffeomorphism $F_0: W_0^1\to W_0^2$. Therefore, the desired diffeomorphism $F$ is produced by 
$$F=f^2\circ f_{i_0-1}^2\circ \cdots f_0^2\circ F_0\circ (f_0^1)^{-1}\circ\cdots (f_{i_0-1}^1)^{-1}\circ (f^1)^{-1}: U_1'\to U_2'.$$
Furthermore, from the construction in \cite{ChC1}, the diffeomorphism $F$ is sufficiently close to the GH map if $\epsilon$ is small enough.	Therefore, we finish the proof. 
	\end{proof}

\begin{theorem}[Finite topology]\label{t:finite_top_bodyregion}
There exist $C=C(n,D)$ and $\epsilon(n)$ with the following property. Let $(M^n,g)$ be a complete Riemannian manifold and $U\subset M$ an open subset with $r_{Rei,\epsilon}(x)>r>0$ for all $x\in U$. Assume further that $\diam(U)\le D r$. Then there exists an open set $U'$ with $B_{r/2}(U)\subset U'\subset B_r(U)$ such that $U'$ has at most $C(n,D)$ many diffeomorphism types. 
\end{theorem}
\begin{remark}
    Basing on Theorem \ref{t:finite_top_bodyregion}, sometimes we will omit the $\epsilon(n)$ in $r_{Rei,\epsilon(n)}(x)$ and write it as $r_{Rei}(x)$. 
\end{remark}
\begin{remark}
If one assumes a lower Ricci curvature bound for the manifold $M$, then one may be possible to prove the above finiteness by using Ricci flow where Perelman's Pseudolocality would play a key role in the proof. (See \cite{ChLee}\cite{LS22}\cite{ST21}\cite{He}) 
\end{remark}

\begin{proof}
The proof follows directly from Gromov's precompactness and Lemma \ref{l:finite_topology_bodyregion}. Actually, consider the space $\mathcal{M}(\epsilon,r,D,n):=\{B_r(U)\subset M: \text{ $U,M$ as in the theorem}\}$. Denote $\overline{\mathcal{M}}(\epsilon,r,D,n)$ to be the closure of $\mathcal{M}(\epsilon,r,D,n)$ under Gromov-Hausdorff topology.  Since $r_{Rei,\epsilon}(x)\ge r>0$ for each $x\in U$, by Gromov's precompactness theorem (see \cite{Gro81}), the set  $\overline{\mathcal{M}}(\epsilon,r,D,n)$ is compact in Gromov-Hausdorff topology. Therefore, we can choose finitely many $\{U_1,U_2,\cdots, U_N\}\subset \overline{\mathcal{M}}(\epsilon,r,D,n)$ with $N\le N(n,D,\epsilon(n))$ such that the $\epsilon(n) r/10$-balls of $U_i$ covers $\overline{\mathcal{M}}(\epsilon,r,D,n)$. For any two $W,V\in T_{\epsilon r/10,d_{GH}}(U_i)\cap \mathcal{M}(\epsilon,r,D,n)$, by Lemma \ref{l:finite_topology_bodyregion}, there exist diffeomorphisms between $V'$ and $U_i'$, between $W'$ and $U_i''$ with $B_{r/2}(V)\subset V'\subset B_r(V)$, $B_{r/2}(W)\subset W'\subset B_r(W)$ and $B_{r/2}(U_i)\subset U_i'\subset B_r(U_i)$, $B_{r/2}(U_i)\subset  U_i''\subset B_r(U_i)$, where $T_{\epsilon r/10,d_{GH}}(U_i)$ is the $\epsilon r/10$-ball of $U_i$ in Gromov-Hausdorff topology. From the proof of Lemma \ref{l:finite_topology_bodyregion}, the subset $U_i'$ and $U_i''$ could be chosen the same for different $V,W\in T_{\epsilon r/10,d_{GH}}(U_i)\cap \mathcal{M}(\epsilon,r,D,n)$.  Due to the finiteness of $U_i$,  we finish the proof. 
\end{proof}

\section{Finite diffeomorphism with bounded $L^{n/2}$ of Riemannian curvature}

In this section, we consider complete Riemannian manifolds $(M^n,g,p)$ satisfying
\begin{align}
\Ric\ge -(n-1), \  \Vol(B_1(p))\ge \rv>0, \ \ \int_{B_8(p)}|\Rm|^{n/2}\le \Lambda.   \label{main-assumptions}
\end{align} 

We will first consider the regularity in a neck region and prove Proposition~\ref{p:neck_diff}. Based on Cheeger-Naber \cite{CheegerNaber_codimension4} we can then prove a decomposition theorem. The finite diffeomorphism type theorem follows by the decomposition theorem.

\subsection{Diffeomorphism of Neck Region}

\begin{lemma}\label{l:cross_section}
Let $(M^n,g,p)$ be a complete manifold satisfying \eqref{main-assumptions}.
For any $\epsilon>0$ there exists $\delta(n,\epsilon,\rv,\Lambda)>0$ such that  if the volume ratio
\begin{align}
|\cV_{r/2}(p)-\cV_2(p)|\le \delta
\end{align}
for some $r\le 1$,  then there exists a space form $\mathbb{S}^{n-1}/\Gamma$ with $|\Gamma|\le C(n,\rv)$ such that  for any $r\le s\le 1$ 
\begin{align}
d_{GH}(B_s(p),B_s(\bar p))\le \epsilon s, ~~\text{ with cone vertex $\bar p\in C(\mathbb{S}^{n-1}/\Gamma)$ }.
\end{align}
\end{lemma}
\begin{remark}\label{r:regularityonneck}
   
   Assuming as Lemma \ref{l:cross_section}, we have  for any $x\in B_{1/2}(p)$ the Reifenberg radius of $x$ has a definite lower bound. To see this, by Lemma \ref{l:cross_section}, the ball $B_{r_0}(x)$ is close to a ball $B_{r_0}(\bar x)\subset C(S^{n-1}/\Gamma)$ which is smooth, where the radius $r_{0}>0$ depending on $\Gamma$ and the ball $B_{r_0}(\bar x)$ is a Euclidean ball.   Hence by Lemma \ref{l:Reifenberg_Colding},  we get the desired Reifenberg radius lower bound. 
\end{remark}

\begin{proof}
For any $\epsilon$, if $|\cV_{r/2}(p)-\cV_2(p)|\le \delta<\delta(n,\epsilon,\rv)$, by Theorem~\ref{t:Almostmetriccone},  we have for any $3r/4\le s\le 1$ the ball $B_{s}(p)$ is close to a cone $C(Z_s)$:  
\begin{align}\label{e:almostcloseconeZs}
d_{GH}(B_s(p),B_s(\bar p_s))\le \epsilon s.
\end{align}
Now we use  Theorem 10.2 b) of Cheeger \cite{Cheeger}
to show  the cross section $Z_s$ is a space form $\mathbb{S}^{n-1}/\Gamma_s$ with $|\Gamma_s|\le C(n,\rv)$. 
We  argue this by contradiction. 
Assume there exists an $\epsilon_0>0$, and a sequence of $M_i^n$ and $r_i\le s_i\le 1$ satisfy $|\cV_{r_i/2}(p_i)-\cV_2(p_i)|\le \delta_i\to 0 $ and \eqref{main-assumptions}, but   $d_{GH}(Z_{s_i},\mathbb{S}^{n-1}/\Gamma)>\epsilon_0$ for any $\Gamma\le O(n)$, where  $Z_{s_i}$ is the cross section in \eqref{e:almostcloseconeZs} with respect to $(M_i^n,p_i)$. By scaling, we may assume $s_i=1$. Up to a subsequence, let us denote $X$ the limit of $B_{1}(p_i)$. By \eqref{e:almostcloseconeZs} we know that the limit is the unit ball in a metric cone with cross section $Z=\lim_{i\to \infty}Z_{s_i}$ and $p_i\to \bar p\in X$ the cone vertex.  Since   $d_{GH}(Z_{s_i},\mathbb{S}^{n-1}/\Gamma)>\epsilon_0$ we get $d_{GH}(Z,\mathbb{S}^{n-1}/\Gamma)\ge \epsilon_0$. By \cite[Theorem 10.2 b)]{Cheeger}, any tangent cone of $X$ is a space form, which implies that the tangent cone at $\bar p$ is a space form. On the other hand, noting that $\bar p$ is the vertex of metric cone $C(Z)$, then we have $C(Z)$ is a space form which contradicts to $d_{GH}(Z,\mathbb{S}^{n-1}/\Gamma)\ge \epsilon_0$ for any $\Gamma$. Here the estimate $|\Gamma|\le C(n,\rv)$ follows directly from the volume lower bound $\Vol(B_1(p))\ge \rv>0$ and volume convergence. 

On the other hand, since the space of cross sections is connected in Gromov-Hausdorff topology, by the rigidity of space form, we have that $\Gamma_s=\Gamma_s'=\Gamma$ for each $3r/4\le s,s'\le 1$. Hence we finish the proof. 
\end{proof}

Now we are ready to prove the following diffeomorphism theorem by slightly modifying Cheeger-Colding's proof (see \cite{Cheeger_finiteness} \cite{AnCh} \cite{ChC1}  ). 
\begin{proposition}\label{p:neck_diff}
Let $(M^n,g,p)$ be a complete manifold satisfying \eqref{main-assumptions}. 
There exists $\delta=\delta(n,\rv, \Lambda)$ such that  if, for any $r\le 1/20$, 
\begin{align}
|\cV_{r/2}(p)-\cV_2(p)|\le \delta, 
\end{align}
 then there exists a diffeomorphism $F: \mathbb{S}^{n-1}/\Gamma\times (r/2,1/2)\to \cN$ where $\Gamma\subset O(n)$ and $A_{11r/20,9/20}(p)\subset \cN\subset A_{9r/20,11/20}(p)$ with $|\Gamma|\le C(n,\rv)$.
 \end{proposition}

\begin{proof}
This proposition essentially follows from Theorem A.1.3 of Cheeger-Colding \cite{ChC1}(see also Remark A.1.47 of Cheeger-Colding \cite{ChC1}). We will just sketch the proof as the proof of Lemma \ref{l:finite_topology_bodyregion}. For more detail, see \cite{ChC1}.

For any $\epsilon>0$ if $\delta\le\delta(n,\rv,\epsilon, \Lambda)$, by Lemma \ref{l:cross_section},  there exists a space form $\mathbb{S}^{n-1}/\Gamma$ with $|\Gamma|\le C(n,\rv)$ such that  for any $r\le s\le 1$ 
\begin{align}
d_{GH}(B_s(p),B_s(\bar p))\le \epsilon s, ~~\text{ with cone vertex $\bar p\in C(\mathbb{S}^{n-1}/\Gamma)$ }
\end{align}
Now we can begin the outline for the construction of diffeomorphism. 
For each $i\ge 0$, we can choose finite subsets $X_i\subset A_{r/5,4/5}(p)\subset M$ and $Y_i\subset A_{r/5,4/5}(\bar p)\subset C(S^{n-1}/\Gamma)$ such that $X_i\cap A_{2^{\ell-1/3,2^{\ell+4/3}}}(p)$ and $Y_i\cap A_{2^{\ell-1/3,2^{\ell+4/3}}}(\bar p)$ are $2^{\ell-i}$-dense in $A_{2^{\ell-1/3,2^{\ell+4/3}}}(p)$ and $A_{2^{\ell-1/3,2^{\ell+4/3}}}(\bar p)$ respectively, for all integer $\ell$ satisfying $r/5 \le 2^{\ell}\le 4/5$. Using such $X_i$ and $Y_i$,  following the same argument as \cite{ChC1},  by constructing and suitably modifying the transition maps, if $\epsilon\le \epsilon(n)$ we can construct a sequence of manifolds $W_i,Z_i$ and diffeomorphism $F_i: W_i\to W_{i+1}$ and $G_i:Z_i\to Z_{i+1}$, and for $i_0$ sufficiently large,  diffeomorphism $F': W_{i_0} \to \cN'$ with $A_{2r/5,3/5}(p)\subset \cN'\subset A_{r/5,4/5}(p)$, and diffeomorphism $G': Z_{i_0}\to \cN_{\Gamma}'$ with $A_{2r/5,3/5}(\bar p)\subset \cN_{\Gamma}'\subset A_{r/5,4/5}(\bar p)$.

Furthermore, as in \cite{ChC1} if $\epsilon\le \epsilon(n)$, there exists a diffeomorphism $H_0: W_0\to Z_0$. Therefore, the desired diffeomorphism $F$ is produced by 
$$F=G'\circ G_{i_0-1}\circ \cdots G_0\circ H_0\circ F_0^{-1}\circ\cdots F_{i_0-1}^{-1}\circ (F')^{-1}: \cN'\to \cN_{\Gamma}'.$$
Furthermore, from the construction in \cite{ChC1}, the diffeomorphism $F$ is sufficiently close to the GH map if $\epsilon$ is small enough.	Therefore, by shrinking the domain $\cN\subset \cN'$ such that $F(\cN)=A_{r/2,1/2}(\bar p)$, we finish the whole proof. 
\end{proof}

\subsection{Decomposition theorem}

Combining the results above and the argument in Cheeger-Naber \cite{CheegerNaber_codimension4} (See also \cite{AnCh},\cite{Cheeger_finiteness}, \cite{Jiang_Naber}, \cite{NaVal} ), we will prove the following decomposition theorem which is the key to our main theorem.  

\begin{theorem}\label{t:decomposition}
Let $(M^n,g)$ satisfy \eqref{main-assumptions} and $\diam(M,g)\le D$. 
Then  we have the following decomposition
\begin{align}
M^n=\cB^1\cup \bigcup_{j_2=1}^{N_2}\cN_{j_2}^2\cup \bigcup_{j_2=1}^{N_2}\cB_{j_2}^2\cup \cdots \cup \bigcup_{j_k=1}^{N_k}\cN_{j_k}^k\cup \bigcup_{j_k=1}^{N_k}\cB_{j_k}^k.
\end{align}
 such that each $\cB$ and $\cN$ are open and satisfy the following:
 \begin{itemize}
 \item[(1)] If $x\in \cB_j^\ell$, then the Reifenberg radius $r_{Rei}(x)\ge r_0(n,\rv,D,\Lambda) \diam(\cB_j^\ell)$, where $r_{Rei}(x)=r_{Rei,\epsilon(n)}(x)$ with $\epsilon(n)$ as in Theorem \ref{t:finite_top_bodyregion}.
 \item[(2)] Each neck $\cN_{j}^\ell$ is diffeomorphic to $\mathbb{R}\times \mathbb{S}^{n-1}/\Gamma_{j}^{\ell}$ where $|\Gamma_{j}^{\ell}|\le C(n,\rv)$.
 \item[(3)] $\cN_j^\ell\cap \cB_j^\ell$ is diffeomorphic to $\mathbb{R}\times \mathbb{S}^{n-1}/\Gamma_{j}^{\ell}$.
 \item[(4)] $\cB_j^{\ell-1}\cap \cN_j^\ell$ are either empty or diffeomorphic to $\mathbb{R}\times \mathbb{S}^{n-1}/\Gamma_{j}^{\ell}$.
 \item[(5)] $N_{\ell}\le N(n,\rv,D,\Lambda)$ and $k\le k(\rv,D,n,\Lambda)$.
 \end{itemize}
\end{theorem}

\begin{proof}
For any given $\delta>0$, we have  $\sup_{x\in M}|\cV_0(x)-\cV_D(x)|\le L\delta$ for some integer $L=L(n,\rv,\delta,D)$. We will construct the decomposition inductively on $L$. Let us now begin the construction and reduce the upper bound $L\delta$ to $(L-1)\delta$.

By Lemma \ref{l:cone_structure}, we have that for any $x\in M$ there exists $1\ge r_x\ge r(n,\rv,\delta)$ that $|\cV_{r_x/100}(x)-\cV_{r_x}(x)|\le \delta$. Choose a Vitali covering $\{B_{r_i^1}(x_i^1)\}_{i=1}^{N_2}$ of $M$ such that 
\begin{itemize}
\item[(a1)] $M=\cup_{i=1}^{N_2}B_{r_i^1}(x_i^1)$ with $r_i^1=\tfrac{r_{x_i^1}}{2}$.
\item[(a2)] $B_{r_i^1/5}(x_i^1)\cap B_{r_j^1/5}(x_j^1)=\emptyset$ for $i\ne j$.
\item[(a3)]  $N_2\le C(n,\rv,\delta,D)$
\end{itemize}

Define the singular scale  for each $y\in \bar B_{r_i^1/50}(x_i^1)$ by 
\begin{align}
s_{y,i}^1=\inf\{s\le r_i^1: \inf_{x\in B_{s/4}(y)}|\cV_{s}(x)-\cV_{2r_i^1}(x)|\ge \delta\}.
\end{align}
By the choice of $x_i^1$ and $r_i^1$ we have $s_{x_i^1,i}^1\le r_i^1/50$. Let $y_i^1\in \bar B_{r_i^1/50}(x_i^1)$ such that 
\begin{align}
s_{y_i^1,i}^1=\inf_{y\in \bar B_{r_i^1/50}(x_i^1)} s_{y,i}^1\le r_i^1/50.
\end{align}
Denote $s_{i}^1=s_{y_i^1,i}^1$. By the definition of the singular scale, we  have the following.
\begin{itemize}
\item[(b1)] There exists $z_i^1\in \bar B_{s_i^1/4}(y_i^1)$ such that $|\cV_{s_i^1}(z_i^1)-\cV_{2r_i^1}(z_i^1)|=\delta$ 
\item[(b2)] For any $x\in \bar B_{r_i^1/50}(x_i^1)$, we have $|\cV_{s_i^1}(x)-\cV_{2r_i^1}(x)|\ge \delta$.
\item[(b3)] $ B_{r_i^1/25}(z_i^1)\subset  B_{r_i^1/10}(x_i^1)$. 
\end{itemize}
Therefore, by (a2) and (b3), we have $B_{r_i^1/25}(z_i^1)\cap B_{r_j^1/25}(z_j^1)=\emptyset$. Define $\cB^1=M\setminus \cup_{i=1}^{N_2}\bar B_{2r_i^1/25}(z_i^1)$ and $\hat{\cN}_{i}^2=B_{r_i^1/4}(z_i^1)\setminus \bar B_{2s_{i}^1}(z_i^1)$.  Hence we arrive at the following decomposition:
\begin{align}
M^n\subset \cB^1\cup \bigcup_{j_2=1}^{N_2}\hat{\cN}_{j_2}^2 \cup \bigcup_{j_2=1}^{N_2} B_{3s_j^1}(z_i^1),
\end{align}
such that 
\begin{itemize}
\item[($1_1$)] $\cB^1\cap \hat{\cN}^2_i=B_{r_i^1/4}(z_i^{1})\setminus \bar B_{2r_i^1/25}(z_i^1)$ and for any $x\in \cB^1$, there exists $z_i^1$ such that $x\in B_{r_i^1}(z_i^1)\setminus \bar B_{2r_i^1/25}(z_i^1)$ with $|\cV_{r_i^1/50}(z_i^1)-\cV_{2r_i^1}(z_i^1)|\le \delta$ and $r_i^1\ge r(n,\rv,\delta,D) \diam(\cB^1)$.
\item[($2_1$)] $\hat{\cN}_i^2\subset B_{2r_i^1}(z_i^1)\setminus \bar B_{s_{i}^1}(z_i^1)$ satisfies $|\cV_{s_i^1}(z_i^1)-\cV_{2r_i^1}(z_i^1)|=\delta$.
\item[($3_1$)] For any $\sup_{x\in B_{2s_i^1}(z_i^1)}|\cV_{0}(x)-\cV_{s_i^1}(x)|\le (L-1)\delta$.
\item[($4_1$)] $N_2\le C(n,\rv,\delta,D)$.
\end{itemize}

Let us now decompose the ball $B_{2s_i^1}(z_i^1)$ using similar argument as above. For any $x\in B_{2s_i^1}(z_i^1)$, by Lemma \ref{l:cone_structure} there exists radius $s_i^1/4\ge r_x^2\ge r(n,\rv,\delta)s_i^1$ such that $|\cV_{r_x^2/100}(x)-\cV_{r_x^2}(x)|\le \delta$.
Choose a Vitali covering $\{B_{r_{\alpha,i}^2}(x_{\alpha,i}^2)\}_{{\alpha,i}=1}^{N_2'}$ of $B_{2s_i^1}(z_i^1)$ such that 
\begin{itemize}
\item[(a1')] $B_{2s_i^1}(z_i^1)\subset \cup_{{\alpha}=1}^{N_2'}B_{r_{\alpha,i}^2}(x_{\alpha,i}^2)$ with $r_{\alpha,i}^2:=\tfrac{r^2_{x_{\alpha,i}^2}}{2}$ and $x_{\alpha,i}^2\in B_{2s_i^1}(z_i^1)$.
\item[(a2')] $B_{r_{\alpha,i}^2/5}(x_{\alpha,i}^2)\cap B_{r_{\beta,i}^2/5}(x_{\beta,i}^2)=\emptyset$ for ${\alpha}\ne \beta$.
\item[(a3')]  $N_2'\le C(n,\rv,\delta,D)$
\end{itemize}
Define the singular scale  for each $y\in \bar B_{r_{\alpha,i}^2/50}(x_{\alpha,i}^2)$ by 
\begin{align}
s_{y,{\alpha,i}}^2=\inf\{s\le r_{\alpha,i}^2: \inf_{x\in B_{s/4}(y)}|\cV_{s}(x)-\cV_{2r_{\alpha,i}^2}(x)|\ge \delta\}.
\end{align}
Using the same argument as above, we have $y_{\alpha,i}^2,s_{\alpha,i}^2,z_{\alpha,i}^2$ satisfying
\begin{itemize}
\item[(b1')] $y_{\alpha,i}^2\in   \bar B_{r_{\alpha,i}^2/50}(x_{\alpha,i}^2)$, $z_{\alpha,i}^2\in \bar B_{s_{\alpha,i}^2/4}(y_{\alpha,i}^2)$ such that $|\cV_{s_{\alpha,i}^2}(z_{\alpha,i}^2)-\cV_{2r_{\alpha,i}^2}(z_{\alpha,i}^2)|=\delta$ 
\item[(b2')] For any $x\in \bar B_{r_{\alpha,i}^2/50}(x_{\alpha,i}^2)$, we have $|\cV_{s_{\alpha,i}^2}(x)-\cV_{2r_{\alpha,i}^2}(x)|\ge \delta$.
\item[(b3')] $ B_{r_{\alpha,i}^2/25}(z_{\alpha,i}^2)\subset  B_{r_{\alpha,i}^2/10}(x_{\alpha,i}^2)$. 
\end{itemize}
Therefore, by (a2') and (b3'), we have $B_{r_{\alpha,i}^2/25}(z_{\alpha,i}^2)\cap B_{r_{\beta,i}^2/25}(z_{\beta,i}^2)=\emptyset$. Define $\cB^2_i=B_{3s_i^1}(z_i^1)\setminus \cup_{{\alpha,i}=1}^{N_2'}\bar B_{2r_{\alpha,i}^2/25}(z_{\alpha,i}^2)$ and $\hat{\cN}_{{\alpha,i}}^3=B_{r_{\alpha,i}^2/4}(z_{\alpha,i}^2)\setminus \bar B_{2s_{{\alpha,i}}^2}(z_{\alpha,i}^2)$.  Hence we arrive at the following decomposition:
\begin{align}
M^n\subset \cB^1\cup \bigcup_{j_2=1}^{N_2}\hat{\cN}_{j_2}^2 \cup \bigcup_{j_2=1}^{N_2}\cB_{j_2}^2\cup \bigcup_{j_2=1}^{N_2}\bigcup_{{\alpha}=1}^{N_2'}\hat{\cN}_{{\alpha,j_2}}^3\cup \bigcup_{j_2=1}^{N_2}\bigcup_{{\alpha}=1}^{N_2'} B_{3s_{\alpha,j_2}^2}(z_{\alpha,j_2}^2).
\end{align}
Rewrite it as 
\begin{align}
M^n\subset \cB^1\cup \bigcup_{j_2=1}^{N_2}\hat{\cN}_{j_2}^2 \cup \bigcup_{j_2=1}^{N_2}\cB_{j_2}^2\cup \bigcup_{j_3=1}^{N_3}\hat{\cN}_{j_3}^3\cup \bigcup_{j_3=1}^{N_3} B_{3s_{j_3}^2}(z_{j_3}^2).
\end{align}
Moreover, the decomposition satisfies
\begin{itemize}
\item[($1_2$)] $\cB^{\ell}_{j}\cap \hat{\cN}^{\ell+1}_i=B_{r_i^\ell/4}(z_i^{\ell})\setminus \bar B_{2r_i^\ell/25}(z_i^\ell)$ or empty, for any $x\in \cB^\ell_i$ there exists $z_i^\ell$ such that $x\in B_{r_i^\ell}(z_i^\ell)\setminus \bar B_{2r_i^\ell/25}(z_i^\ell)$ with $|\cV_{r_i^\ell/50}(z_i^\ell)-\cV_{2r_i^\ell}(z_i^\ell)|\le \delta$ and $r(n,\rv,\delta,D)\diam(\cB^\ell_i)\le r_i^\ell$  for $\ell=1,2$.
\item[($2_2$)] $\hat{\cN}_i^\ell\subset B_{2r_i^{\ell-1}}(z_i^{\ell-1})\setminus \bar B_{s_{i}^{\ell-1}}(z_i^{\ell-1})$, $\hat{\cN}_i^\ell\cap \cB_i^\ell= B_{3s_i^{\ell-1}}(z_i^{\ell-1})\setminus \bar B_{2s_{i}^{\ell-1}}(z_i^{\ell-1})$ satisfies $|\cV_{s_i^{\ell-1}}(z_i^{\ell-1})-\cV_{2r_i^{\ell-1}}(z_i^{\ell-1})|=\delta$ for $\ell=2,3$.
\item[($3_3$)] $\sup_{x\in B_{2s_{j}^2}(z_{j}^2)}|\cV_{0}(x)-\cV_{s_{j}^2}(x)|\le (L-2)\delta$.
\item[($4_4$)] $N_2,N_3\le C(n,\rv,\delta,D)$.
\end{itemize}
Proceeded inductively as above $k\le L$ times, we will get 

\begin{align}
M^n=\cB^1\cup \bigcup_{j_2=1}^{N_2}\hat{\cN}_{j_2}^2\cup \bigcup_{j_2=1}^{N_2}\cB_{j_2}^2\cup \cdots \cup \bigcup_{j_k=1}^{N_k}\hat{\cN}_{j_k}^k\cup \bigcup_{j_k=1}^{N_k}\cB_{j_k}^k.
\end{align}
such that 
\begin{itemize}
\item[($1_k$)] $\cB^{\ell}_{j}\cap \hat{\cN}^{\ell+1}_i=B_{r_i^\ell/4}(z_i^{\ell})\setminus \bar B_{2r_i^\ell/25}(z_i^\ell)$ or empty, for any $x\in \cB^\ell_i$ there exists $z_i^\ell$ such that $x\in B_{r_i^\ell}(z_i^\ell)\setminus \bar B_{2r_i^\ell/25}(z_i^\ell)$ with $|\cV_{r_i^\ell/50}(z_i^\ell)-\cV_{2r_i^\ell}(z_i^\ell)|\le \delta$ and $r(n,\rv,\delta,D)\diam(\cB^\ell_i)\le r_i^\ell$  for $\ell=1,2,\cdots, k-1$.
\item[($2_k$)] $\hat{\cN}_i^\ell\subset B_{2r_i^{\ell-1}}(z_i^{\ell-1})\setminus \bar B_{s_{i}^{\ell-1}}(z_i^{\ell-1})$, $\hat{\cN}_i^\ell\cap \cB_i^\ell= B_{3s_i^{\ell-1}}(z_i^{\ell-1})\setminus \bar B_{2s_{i}^{\ell-1}}(z_i^{\ell-1})$ satisfies $|\cV_{s_i^{\ell-1}}(z_i^{\ell-1})-\cV_{2r_i^{\ell-1}}(z_i^{\ell-1})|=\delta$ for $\ell=2,3,\cdots, k$.
\item[($3_k$)]  $\sup_{x\in \cB_{j_k}^k}|\cV_{0}(x)-\cV_{s_{j_k}^k}(x)|\le \delta$ where $s_{j_k}^k\ge  \diam(\cB_{j_k}^k)/10$.
\item[($4_k$)] $N_2,N_3,\cdots, N_k\le C(n,\rv,\delta,D)$.
\item[($5_k$)] $k\le L=L(n,\rv,\delta,D)$.
\end{itemize}
Now the theorem follows directly from Lemma \ref{l:Reifenberg_Colding} and Proposition \ref{p:neck_diff} by fixing small $\delta=\delta(n,\rv,\Lambda)$ and choosing $\cN_i^\ell\subset \hat{\cN}_i^\ell$ by Proposition \ref{p:neck_diff} such that $\cN_i^\ell$ is diffeomorphic to $ \mathbb{R} \times \mathbb{S}^{n-1}/\Gamma_{i}^{\ell}$. See also Remark \ref{r:regularityonneck} for the regularity on the neck. 
\end{proof}

\subsection{Proving Theorem \ref{t:maintheorem}}

Let $(M^n,g)\in \mathcal{M}(n,\rv,D,\Lambda)$. By Theorem \ref{t:decomposition}, there exists a decomposition 
\begin{align}
M^n=\cB^1\cup \bigcup_{j_2=1}^{N_2}\cN_{j_2}^2\cup \bigcup_{j_2=1}^{N_2}\cB_{j_2}^2\cup \cdots \cup \bigcup_{j_k=1}^{N_k}\cN_{j_k}^k\cup \bigcup_{j_k=1}^{N_k}\cB_{j_k}^k.
\end{align}
 such that each $\cB$ and $\cN$ are open and satisfy the following:
 \begin{itemize}
 \item[(1)] If $x\in \cB_j^\ell$, then the Reifenberg radius $r_{Rei}(x)\ge r_0(n,\rv,D,\Lambda) \diam(\cB_j^\ell)$.
 \item[(2)] Each neck $\cN_{j}^\ell$ is diffeomorphic to $\mathbb{R}\times \mathbb{S}^{n-1}/\Gamma_{j}^{\ell}$ where $\Gamma_j^\ell\subset O(n)$ and $|\Gamma_{j}^{\ell}|\le C(n,\rv)$.
 \item[(3)] $\cN_j^\ell\cap \cB_j^\ell$ is diffeomorphic to $\mathbb{R}\times \mathbb{S}^{n-1}/\Gamma_{j}^{\ell}$.
 \item[(4)] $\cB_j^{\ell-1}\cap \cN_j^\ell$ are either empty or diffeomorphic to $\mathbb{R}\times \mathbb{S}^{n-1}/\Gamma_{j}^{\ell}$.
 \item[(5)] $N_{\ell}\le N(n,\rv,D,\Lambda)$ and $k\le k(\rv,D,n,\Lambda)$.
 \end{itemize}

By Theorem \ref{t:finite_top_bodyregion}, the body region $\cB_j^\ell$ has at most $C(\rv,n,D,\Lambda)$-diffeomorphism types. By $|\Gamma_j^\ell|\le C(n,\rv)$ and $\Gamma_j^\ell\subset O(n)$, the neck region has at most $C(n,\rv)$-diffeomorphism types. By (3), (4) and $|\Gamma_j^\ell |\le C(n,\rv)$, the intersection component of neck regions and body regions has at most $C(n,\rv)$-diffeomorphism types. From the proof of Theorem \ref{t:decomposition} and the construction of diffeomorphism in Proposition \ref{p:neck_diff}, we can suppose that the induced attaching map between intersection components is sufficiently close to being an isometry of space form $\mathbb{S}^{n-1}/\Gamma_j^\ell$ in pointwise sense. This implies that the induced attaching map is isotopic  to such an isometry. On the other hand, noting that the isometry group of a compact manifold is a Lie group with finitely many component, hence there are only finitely many isotopy classes of such attaching maps for each neck regions. Therefore, there are at most $C(\rv,n,D,\Lambda)$ many ways attaching all the body regions and necks. Hence we get at most $C(\rv,n,D,\Lambda)$-diffeomorphism types. This finishes the proof of the theorem. \qed

\section{Diffeomorphism for K\"ahler manifold}
In this section, we will prove the following finitely many homeomorphism types theorem (Theorem \ref{t:Kahler_homeomorphism}) for K\"ahler surfaces. By using Chern-Gauss-Bonnet formula or Chern-Weil theory, we can get bounded $L^2$ Riemann curvature. Therefore Theorem \ref{t:maintheorem} and Theorem \ref{t:Kahler_homeomorphism} imply the finite diffeomorphism types, Theorem \ref{t:Kahler_diffeomorphism}.

\begin{theorem}\label{t:Kahler_homeomorphism}
Let $(M^4,g)$ be compact K\"ahler manifold with $\Ric\ge -(n-1)$, $\Vol(M)\ge \rv>0$, $\diam(M)\le D$ and $\int_{M}|R|^2\le \Lambda$. Then $(M^4,g)$ has at most $C(\rv,D,\Lambda)$ many homeomorphism types. As a consequence, we have $\int_M|\Rm|^2\le C(\rv,D,\Lambda)$.
\end{theorem}

Under the condition of Theorem \ref{t:Kahler_homeomorphism}, we don't know whether every tangent cone of the limit space of such sequence is unique and flat. Instead, we can prove that every tangent cone is a cone with  smooth cross section and all cross sections have the same topology. This is good enough for us to glue local homeomorphisms to get a global homeomorphism on the neck region.

In the proof of Theorem \ref{t:Kahler_homeomorphism}, comparing with Theorem \ref{t:maintheorem}, the main difference is the cross section Lemma \ref{l:cross_section} and neck diffeomorphism Proposition \ref{p:neck_diff}. Let us begin analogous lemmas for the proof of Theorem \ref{t:Kahler_homeomorphism}.

\subsection{$\epsilon$-regularity for K\"ahler manifolds.}
The following $\epsilon$-regularity could be found in \cite{Cheeger} 
and Lemma 5.2 of Tian-Wang \cite{TianWang}.   Actually, by Cheeger-Colding-Tian's splitting theorem in \cite{CCTi_eps_reg} for K\"ahler manifold, if a ball is close to a splitting ball with Euclidean fact $\mathbb{R}^{n-3}$, then the ball must be close to ball splitting $\mathbb{R}^{n-2}$ and $n-2$ is even. Hence we have
\begin{lemma}[\cite{Cheeger,TianWang}]\label{l:eps_regular_Kahler}
Let $(M^n,g,p)$ be a K\"ahler manifold with real dimension $n$ and $\Vol(B_1(p))\ge \rv>0$. For any $\epsilon>0$ if $\delta\le \delta(n,\rv,\epsilon)$, $\Ric\ge -\delta$, \ $\fint_{B_2(p)}|\Ric|\le \delta$ and 
\begin{align}
d_{GH}(B_2(p), B_2(0^{n-3},y_c))\le \delta,~~~\text{ $(0^{n-3},y_c)$ is a cone vertex of metric cone $\mathbb{R}^{n-3}\times C(Y)$},
\end{align}
then 
\begin{align}
d_{GH}(B_1(p),B_1(0^n))\le \epsilon.
\end{align}

\end{lemma}
As a direct consequence, we have
\begin{corollary}\label{c:reifenberglowerKahler}
Let $(M^4,g,p)$ be a K\"ahler manifold with lower Ricci curvature $\Ric\ge -\delta$ and $\Vol(B_1(p))\ge \rv>0$. For any $\epsilon>0$ if $\delta \le\delta(\rv,\epsilon)$ and $\int_{B_2(p)}|R|^2\le \delta$ and 
\begin{align}
d_{GH}(B_2(p),B_2(\bar p))\le \delta,  \text{ with metric cone $(C(X),\bar p)$}
\end{align}
then the Reifenberg radius $r_{Rei,\epsilon}(X)\ge r_0(\rv,\epsilon)>0$, where the Reifenberg radius is defined in Definition \ref{d:DefReifenbergradius}.
\end{corollary}
\begin{proof}
 For any $q\in \partial B_1(p)$ and any $\epsilon'>0$, by Lemma \ref{l:cone_structure}, there exists $r>\bar r(\epsilon',\rv)$ that $B_r(q)$ is $\epsilon' r$-close to a metric cone. Noting that there exists a ray from the cone vertex, by the almost splitting Theorem \ref{t:Almostsplitting} if $\delta\le \delta(\epsilon',\rv)$ we have that $B_r(q)$ is $\epsilon' r$-close to a metric cone which splits a factor $\mathbb{R}$. On the other hand, noting that $r^2\fint_{B_r(q)}|R|\le C(\rv)\left(\int_{B_r(q)}|R|^2\right)^{1/2}\le C(\rv)\left(\int_{B_2(p)}|R|^2\right)^{1/2}\le C(\rv)\delta^{1/2}$.  For any $\epsilon''>0$ if $\epsilon'\le \epsilon'(\rv,\epsilon'')$ and $\delta\le \delta(\rv,\epsilon'')$ we have by Lemma \ref{l:eps_regular_Kahler} that 
$B_{r/2}(q)$ is $\epsilon'' r$-close to $B_{r/2}(0^4)$. For any $\eta>0$, if $\delta\le \delta(\rv,\epsilon',\epsilon'',\eta)$,  this implies for some $s\ge s(\rv,\bar r)>0$ and any $x\in X$, 
\begin{align}
d_{GH}(B_s(x),B_s(0^3))\le \eta s, \text{ with $0^3\in \mathbb{R}^3$}.
\end{align}
From \cite{Ketcone} we know that $\Ric_X\ge 2$ in the sense of RCD. Thus by \cite{DPG18} we have for each $0<t<s$ that $B_t(x)\subset X$ is $\eta' t$-close to $B_t(0^3)$ if $\eta\le \eta(\rv,\eta')$. Hence we have finished the proof by choosing $\eta'\le \eta'(\rv,\epsilon)$ and $\delta\le \delta(\rv,\epsilon)$.
\end{proof}

In the proof of Corollary \ref{c:reifenberglowerKahler}, we only need the smallness of $L^2$-Ric on the annulus but not on the whole ball, hence we can directly get the following
\begin{proposition}[$\epsilon$-regularity]\label{p:eps_Kahler}
Let $(M^4,g,p)$ be a K\"ahler manifold with $\Vol(B_1(p))\ge \rv>0$. For any $\epsilon>0$, there exists $\delta(\rv,\epsilon)$ and $r_0(\rv,\epsilon)>0$ such that if $\Ric\ge -\delta$ and 
\begin{itemize}
\item[(1)] $d_{GH}(B_2(p),B_2(x_c))\le \delta$, for some metric cone $(C(X),x_c)$
\item[(2)] $\int_{B_2(p)\setminus \bar B_{1/4}(p)}|R|^2\le \delta$,
\end{itemize}
then for any $x\in B_{3/2}(p)\setminus \bar B_{1/2}(p)$, we have 
\begin{align}
d_{GH}(B_{r_0}(x),B_{r_0}(0^4))\le \epsilon r_0,
\end{align}
and the Reifenberg radius $r_{Rei,\epsilon}(X)\ge r_0(\rv,\epsilon)>0$,
\end{proposition}

\subsection{Diffeomorphism of Neck Region}

In this subsection,  we prove the homeomorphism of neck region. In order to do this, we use the {\it modified monotone quantity} \eqref{e:VR montone}: 
\begin{align} 
\cVR_r(x):=\cV_r(x)+\int_{B_r(x)}|R|^{2}, ~~x\in M, r>0.
\end{align}
where $\cV_r(x)=-\log \frac{\Vol(B_r(x))}{\Vol_{-1}(B_r)}$. Note that $\cVR_r(x)$ is monotone with respect to $r$ for any fixed $x\in M$ and $\cVR_0(x)=0$ and $\cVR_1(x)\le C(\rv,\Lambda)$ for any $x\in B_1(p)$.

\begin{lemma}\label{l:cross_section_Kahler}
Let $(M^4,g,p)$ be a K\"ahler manifold with lower Ricci curvature $\Ric\ge -3$ and $\Vol(B_1(p))\ge \rv>0$. For any $\epsilon>0$ there exists $\delta(\epsilon,\rv)$ such that  if 
\begin{align}
|\cVR_{r/2}(p)-\cVR_2(p)|\le \delta
\end{align}
for some $r\le 1$,  then for any $r\le s\le 1$ there exists $3$-smooth manifold $(Z_s,d_s)$ such that
\begin{itemize}
\item[(a)] $d_{GH}(B_s(p),B_s(\bar p_s))\le \epsilon s$, ~~ with cone vertex $\bar p_s\in C(Z_s)$ .
\item[(b)] $r_{Rei}(Z_s,d_s)\ge r_0(n,\rv)>0$.
\item[(c)] the cross section $(Z_s,d_s)$ is diffeomorphic to each other and $Z_s$ has at most $C(\rv)$ many diffeomorphism types. 
\item[(d)] there exists diffeomorphism $F_s: A_{s/100,s}(\bar p_s)\subset C(Z_s)\to M$ such that $F_s$ is an $\epsilon s$-GH map to $A_{s/100,s}(p)$. 
\end{itemize}
\end{lemma}

\begin{proof}
For any $\epsilon$, if $|\cVR_{r/2}(p)-\cVR_2(p)|\le \delta<\delta(\epsilon,\rv)$ we have by Cheeger-Colding's almost metric cone Theorem \ref{t:Almostmetriccone} that for any $r/2\le s\le 1$ the ball $B_{s}(p)$ is $\epsilon s$ close to a cone $C(Z_s)$:
\begin{align}
d_{GH}(B_s(p),B_s(\bar p_s))\le \epsilon s.
\end{align}
Since  $\int_{A_{r/2, 2}(p)}|R|^{2}\le \delta$, by Proposition \ref{p:eps_Kahler} we can choose $Z_s$ to be smooth and the Reifenberg radius $r_{Rei,\epsilon}(Z_s)\ge r_0(\epsilon,\rv)>0$.  The Reifenberg radius lower bound implies that $(Z_s,d_s)$ has at most $C(\epsilon,\rv)$ many diffeomorphism types by Lemma \ref{l:finite_homeo}.

For the diffeomorphism between $Z_s$ and $Z_{s'}$, note that if $|s-s'|\le \epsilon^2 s$ then $d_{GH}(Z_s,Z_{s'})\le 10\epsilon $. Applying Proposition \ref{p:homeo_XY} we have that $Z_s$ is diffeomorphic to $Z_{s'}$. Therefore, for any $3 r/4\le s,s'\le 1$, $Z_s$ is diffeomorphic to $Z_{s'}$. For any fixed $s$, the last statement (d) holds directly by Proposition \ref{p:homeo_XY}.
\end{proof}

Let us first recall a deformation lemma from 
Theorem 5.4 of Siebenmann \cite{Sie} or Lemma 4.7 of Kapovitch \cite{Kapovitch}.
\begin{lemma}[Deformation] \label{l:deformation}
Let $(X^{n-1},d_X)$ be a closed  topological manifold with metric $d_X$ satisfying $\diam(X)\le D$ and Reifenberg radius $r_{Rei}\ge a>0$. For any $\epsilon>0$, there exists $\delta(n,D,a,\epsilon)>0$ 
such that the following $\mathcal{D}({X,\delta,\epsilon,0,3/4})$ property holds. 

\noindent
$\mathcal{D}(X,{\delta,\epsilon,0,3/4})$: Let $f: X\times (-1,1)\to X\times \mathbb{R}$ be an embedding and $\delta$-close to the inclusion $X\times (-1,1)\subset X\times \mathbb{R}$, then there exists deformation $\tilde{f}: X\times (-1,1)\to X\times \mathbb{R}$ such that 
\begin{itemize}
\item[(1)] $\tilde{f}$ is an embedding and $\epsilon$-close to the inclusion.
\item[(2)] $\tilde{f}=Id$ in a neighborhood of $X\times \{0\}$.
\item[(3)] $\tilde{f}=f$ in $X\times (-1,-3/4)$ and $X\times (3/4,1)$.  
\end{itemize}
\end{lemma}
\begin{proof}
By Lemma 4.7 of \cite{Kapovitch} or Theorem 5.4 of Siebenmann \cite{Sie}, $\mathcal{D}(X,{\delta_X,\epsilon,0,3/4})$ always holds for $X$ with some  $\delta_X>0$. It suffices to show that $\delta_X$ only depends on $n,D,a,\epsilon$. To see this, we will argue by contradiction and use Gromov's precompactness theorem. Assume there exists $\epsilon_0>0$ and $(X_i,d_i)$ satisfies $\diam(X_i)\le D$, $r_{Rei}(X_i)\ge a>0$ such that $\mathcal{D}(X_i,i^{-1},\epsilon_0,0,3/4)$ fails.  By Gromov's precompactness theorem, up to a subsequence, $(X_i,d_i)\to (X_\infty,d_\infty)$ which still satisfies $\diam(X_\infty)\le D$, $r_{Rei}(X_\infty)\ge a>0$. For $i$ sufficiently large, by Proposition \ref{p:homeo_XY} there exists homeomorphism $\Phi_i: X_i\to X_\infty$ which is also an $\eta_i$-GH map with $\eta_i\to 0$.  In $X_\infty$ by Lemma 4.7 of \cite{Kapovitch} there exists $\bar \delta(X_\infty)$ such that $\mathcal{D}(X_\infty,\bar \delta,\epsilon_0/2,0,1/2)$ holds. Let us show that $\mathcal{D}(X_i,\bar \delta/10,\epsilon_0,0,3/4)$ holds for $i$ sufficiently large. Actually, let $f_i: X_i\times (-1,1)\to X_i\times \mathbb{R}$ be an embedding $\bar \delta/10$-close to the inclusion $X_i\times (-1,1)\subset X_i\times \mathbb{R}$. Thus for $i$ large enough, $g_i=(\Phi_i,t)\circ f_i\circ (\Phi_i^{-1},t): X_\infty \times (-1,1)\to X_\infty \times \mathbb{R}$ is an embedding $\bar \delta/2$ close to the inclusion. Therefore, there exists $\tilde{g}_i: X_\infty \times (-1,1)\to X_\infty \times \mathbb{R}$ such that 
\begin{itemize}
\item[(1)] $\tilde{g}_i$ is an embedding and $\epsilon_0/2$-close to the inclusion.
\item[(2)] $\tilde{g}_i=Id$ in a neighborhood of $X_\infty\times \{0\}$.
\item[(3)] $\tilde{g}_i=g_i$ in $X_\infty\times (-1,-1/2)$ and $X_\infty\times (1/2,1)$.  
\end{itemize}
The map $\tilde{f}_i=(\Phi_i^{-1},t)\circ \tilde{g}_i \circ (\Phi_i,t): X_i\times (-1,1)\to X_i\times \mathbb{R}$ satisfies the properties of deformation. Thus the property $\mathcal{D}(X_i,\bar \delta/10,\epsilon_0,0,3/4)$ holds which is a contradiction. Therefore, we have finished the proof. 
\end{proof}

Based on the above deformation lemma, we can glue diffeomorphism of (d) in Lemma \ref{l:cross_section_Kahler} of different scales to  get homeomorphism on the whole neck. 
\begin{proposition}\label{p:neck_homeo}
Let $(M^4,g,p)$ be a K\"ahler manifold with lower Ricci curvature $\Ric\ge -3$ and $\Vol(B_1(p))\ge \rv>0$. Then for any $\epsilon>0$ there exists $\delta=\delta(\rv,\epsilon)$  such that for any $r\le 1/20$  if 
\begin{align}
|\cVR_{r/2}(p)-\cVR_2(p)|\le \delta
\end{align}
then there exists a smooth manifold $X$ which has at most $C(\rv)$ many homeomoprhism types and a Riemannian metric $g_r$ on $X\times (r,1/2)$ such that 
\begin{itemize}
\item[(a)] there exists homeomorphism $F: X\times (r/10,1)\to \cN$ where  $A_{r/10,1/2}(p)\subset \cN\subset A_{r/20,6/10}(p)$.
\item[(b)] $F: (X\times (s/10,s),g_r)\to M$ is $\epsilon s$-GH map for each $r\le s\le 1/2$.
\end{itemize}
 \end{proposition}
\begin{proof}
The proposition follows by a standard gluing argument based on the deformation Lemma \ref{l:deformation} and local diffeomorphism Lemma \ref{l:cross_section_Kahler}. By Lemma \ref{l:cross_section_Kahler}, for any $\delta'>0$, if $\delta\le \delta(n,\delta',\rv)$ for any $r/2\le s\le 2/3$, there exists diffeomorphism $F_s: ((X\times (s/100,s),g_s)\to M$ which is an $\delta' s$-GH map to $A_{s/100,s}(p)\subset M$. Let $s_i=10^{-i}2/3$ and $F_i=F_{s_i}$. We will glue $F_i$ to produce $F$ by using  Lemma \ref{l:deformation}. It suffices to build a homeomorphism $\hat{F}_i$ satisfying the following for each $i$ such that $r/2\le s_i\le 2/3$.  We will build  $\hat{F}_i: X\times (s_{i+1}/100, s_0)\to M$ inductively such that for any $\epsilon>0$ if $\delta\le \delta(n,\rv,\epsilon)$ we have
\begin{itemize}
\item[i.1] There exists metric $\hat{g}_i$ on $X\times (s_{i}/100, s_0)$ such that $\hat{g}_i=g_i$ on $X\times (s_{i}/100,s_{i}/10)$.
\item[i.2] The map $\hat{F}_i: (X\times (s_{j}/100,s_j),\hat{g}_i)\to M$ is an $\epsilon s_j$-GH map to $A_{s_j/100,s_j}(p)$ for each $j\le i$.
\item[i.3] $\hat{F}_i=F_i$ on $X\times (s_{i}/100, s_{i}/10)$ up to a diffeomorphism on $X$ which does not change the $GH$-closeness.
\end{itemize}
Define $\hat{F}_0=F_0$. Assume $\hat{F}_i$ has been built and let us construct $\hat{F}_{i+1}$. It suffices to define a metric $\hat{g}_{i+1}$ and glue $F_{i+1}$ to $\hat{F}_i$ such that $\hat{g}_{i+1}=\hat{g}_i$ on $(s_{i}/10,s_0)$. 
Actually, it suffices to glue $F_{i}$ and $F_{i+1}$ on $X\times (s_{i}/100, s_{i}/10)$.  By Lemma \ref{l:cross_section_Kahler}, we have  diffeomorphism $\Psi_{s_i,s_{i+1}}: \Big(X\times (s_{i}/80,s_{i}/20), g_{s_{i+1}}\Big) \to \Big(X\times (s_{i}/80,s_{i}/20), g_{s_{i}}\Big)$.
Since $F_i, F_{i+1}$ are $\delta' s_i$-GH maps, the map $H_i:=F_{i+1}^{-1}\circ F_i \circ \Psi_{s_i,s_{i+1}}: \Big(X\times (s_{i}/80,s_{i}/20), g_{s_{i+1}}\Big)\to \Big(X\times \mathbb{R}, g_{s_{i+1}}\Big)$ is well defined and $C(n)\delta' s_i$ close to the inclusion $(X\times (s_{i}/80,s_{i}/20)\subset X\times \mathbb{R}$.    

By Lemma \ref{l:deformation}, for any $\delta''$ if $\delta'\le \delta'(n,\rv,\delta'')$ there exists embedding $\tilde{H}_i: (X\times (s_{i}/80,s_{i}/20)\to X\times \mathbb{R}$ such that
\begin{itemize}
\item[(1)] $\tilde{H}_i$ is an embedding and $\delta'' s_{i+1}$-close to the inclusion.
\item[(2)] $\tilde{H}_i=Id$ in a neighborhood of $X\times \{s_{i}/50\}$.
\item[(3)] $\tilde{H}_i=H_i$ in $X\times (s_{i}/80,s_{i}/70)$ and $X\times (s_{i}/40,s_{i}/20)$.  
\end{itemize} 
Let $\hat{H}_i=\tilde{H}_i$ in $X\times (s_{i}/50,s_i/20)$ and $\hat{H}_i=Id$ in $X\times (s_{i+1}/100,s_{i}/50)$. 

Let us now define the gluing map $\hat{F}_{i+1}: X\times (s_{i+1}/100, s_0)\to M$ by 
$$
\hat{F}_{i+1}:=\left\{\begin{array}{cc}
\hat{F}_i(z), &  z\in X\times (s_{i}/80,s_0)\\
F_{i+1}\circ \hat{H}_i(z)\circ \Psi_{s_i,s_{i+1}}^{-1},& z\in X\times (s_{i+1}/100,  s_{i}/70)
\end{array}\right.
$$
Consider the gluing metric $\hat{g}_{i+1}=\varphi \hat{g}_{i}+(1-\varphi)(\Psi_{s_i,s_{i+1}}^{-1})^*g_{i+1}$ on $X\times (s_{i+1}/100, s_0)$ where $\varphi$ is a smooth cut-off function such that $\varphi\equiv 1$ on $X\times (s_{i}/80,s_0)$ and $\varphi\equiv 0$ on $X\times (s_{i+1}/100,s_{i}/70)$. It is easily check that the map $\hat{F}_{i+1}$ and $\hat{g}_{i+1}$ satisfy $(i+1).1-(i+1).3$ if $\delta''\le \delta''(\epsilon)$. Thus we finish the proof.
\end{proof}

\noindent
\textbf{Proof of Theorem \ref{t:Kahler_homeomorphism}}
 The argument is similar with the proof of Theorem \ref{t:maintheorem} by using the above lemmas. Using the same argument as Theorem \ref{t:decomposition}, we can deduce a decomposition theorem in the K\"ahler case. The remaining argument is similar with the proof of Theorem \ref{t:maintheorem}. $\qed$

\bibliographystyle{plain}

\end{document}